\documentclass[11pt]{amsart}

\usepackage{amsmath,amsfonts,amssymb,enumerate}
\usepackage{amsthm,dsfont}
\usepackage{latexsym, mathtools,verbatim,graphicx}
\usepackage{amscd,hyperref}
\usepackage{bbold}
\usepackage{bbm}
\usepackage{color}
\theoremstyle{plain}
\newtheorem{theorem}{Theorem}[section]
\newtheorem{thm}[theorem]{Theorem}

\newtheorem{corollary}[theorem]{Corollary}
\newtheorem{lem}[theorem]{Lemma}
\newtheorem{prop}[theorem]{Proposition}
\theoremstyle{definition}

\newtheorem{defn}[theorem]{Definition}

\theoremstyle{remark}
\newtheorem{rem}[theorem]{Remark}

\newtheorem*{PNT}{{\rm \textbf{Prime Number Theorem}}}

\newenvironment{demode}
  {\noindent {{\it Proof of }}}%
  {\hfill \fbox{}}

\newcommand{\R}{\mathbb{R}}
\newcommand{\C}{\mathbb{C}}

\newcommand{\D}{\mathbb{D}}
\newcommand{\N}{\mathbb{N}}
\newcommand{\1}{\mathds{1}}
\newcommand{\s}{\sigma}

\DeclareMathOperator*{\Span}{span}
\newcommand{\Sum}{{\rm Sum}}

\begin{document}
\title[Zero-free regions of the Riemann zeta function]{Zero-free regions of the Riemann zeta function and approximation in weighted Dirichlet spaces}
\author[Gallardo-Guti\'errez]{Eva A. Gallardo-Guti\'errez}
\address{Eva A. Gallardo-Guti\'errez \newline Departamento de An\'alisis Matem\'atico y Matem\'atica Aplicada,\newline
Facultad de Matem\'aticas,
\newline Universidad Complutense de
Madrid, \newline
 Plaza de Ciencias N$^{\underbar{\Tiny o}}$ 3, 28040 Madrid,  Spain
 \newline
and Instituto de Ciencias Matem\'aticas ICMAT (CSIC-UAM-UC3M-UCM),
\newline Madrid,  Spain }\email{eva.gallardo@mat.ucm.es}

\author[Seco]{Daniel Seco}
\address{Daniel Seco \newline Departamento de An\'alisis Matem\'atico e IMAULL \newline Universidad de la Laguna \newline  Avenida Astrof\'isico Francisco S\'anchez, s/n.  \newline 38206 San Crist\'obal de La Laguna \newline
Santa Cruz de Tenerife,  Spain} \email{dsecofor@ull.edu.es}

\begin{abstract}
We study zero-free regions of the Riemann zeta function $\zeta$ related to an approximation problem in the weighted Dirichlet space $D_{-2}$ which is known to be equivalent to the Riemann Hypothesis since the work of B\'aez-Duarte. We prove, indeed, that analogous approximation problems for the standard weighted Dirichlet spaces $D_{\alpha}$ when $\alpha \in (-3,-2)$ give conditions so that the half-plane $\{s \in \C: \Re (s) > -\frac{\alpha+1}{2}\}$ is also zero-free for $\zeta$. Moreover, we extend such results to a large family of weighted spaces of analytic functions $\ell^p_{\alpha}$. As a particular instance, in the limit case $p=1$ and $\alpha=-2$, we provide a new equivalent formulation of the Prime Number Theorem.
\end{abstract}

\thanks{We acknowledge financial support by the Spanish Ministry of Science and Innovation through the Plan Nacional grants PID2019-106433GB-I00, PID2019-105979GB-I00, PID2022-137294NB-I00 and PID2023-149061NA-I00. The first author also acknowledges support from ``Severo Ochoa Programme for Centres of Excellence in R\&D'' (``Ayuda extraordinaria a
Centros de Excelencia Severo Ochoa') of the Ministry of Economy and Competitiveness of Spain, funded by the European Regional Development Fund and from the Spanish National Research Council (grants CEX2019-000904-S and 20205CEX001). In addition, the second author is funded by the Ram\'on y Cajal programme from Agencia Estatal de Investigaci\'on through grant RYC2021-034744-I, and by the Madrid Government (Comunidad de Madrid-Spain) under the Multiannual Agreement with UC3M in the line of Excellence of University Professors (EPUC3M23), and in the context of the V PRICIT (Regional Programme of Research and Technological Innovation). }



\date{\today}

\maketitle

\section{Introduction and prelimnaries}

The Riemann zeta function $\zeta$ is a very classical object in Mathematics and its links to various properties in Number Theory is well established nowadays. Given a complex number $s \in \C$ with real part $\Re(s) > 1$, the Riemann zeta function is defined by
\[\zeta(s):=\sum_{n=1}^{\infty} n^{-s}.\]
For other values of $s \in \C \backslash \{1\}$, it admits a unique analytic continuation which is also commonly denoted by $\zeta$.

\smallskip

As shown by Riemann, the localization of the zeros of $\zeta$ has deep connections with the distribution of the prime numbers. In particular, the famous Riemann Hypothesis (RH) can be described in terms of some error bounds for how the logarithmic integral is different from the number of primes up to a number $x$,  $\pi(x)$. Recall that (RH) asserts that those zeros of $\zeta$ that are not negative integers, namely the \emph{nontrivial zeros}, lie on the \emph{critical line}
\[\{s \in \C : \Re(s) = 1/2\}.\]
Although there are some positive evidences of this conjecture, the problem at large remains open.

\smallskip

It is well known that the nontrivial zeros are actually symmetric with respect to the critical line, and in terms of their real parts the only thing known to date is that they lie in the strip  $$\{s \in \C : \Re(s) \in (0,1)\}.$$

A function that is strongly associated with the distribution of the zeros of the Riemann zeta function is the so-called M\"obius function $\mu$. Recall that $\mu$ is the arithmetic function defined on the set of positive integers $\N$ taking values in $\{-1,0,1\}$, namely, $\mu : \N \rightarrow \{-1,0,1\}$, where $\mu(1)=1$, each square-free  factorization natural number $k$ is mapped to $(-1)^{\omega(k)}$ where $\omega(k)$ is the number of prime factors of $k$ and  each $k$ divisible by a square prime, is mapped to 0. We refer to \cite{BateDia} for this and many other related classical topics in Analytic Number Theory.

\smallskip

Our main aim in this work is establishing sufficient conditions on the behavior of sums related to the M\"obius function which guarantee that the Riemann zeta function has zero-free regions of the complex plane $\C$. In this regard, our work is strongly influenced by the classical ones in this context by Nyman \cite{Nym}, Beurling \cite{Beurling1955}, B\'aez-Duarte \cite{BaezDuarte, BaDu1}, Balazard-Saias \cite{BalazardSaias4} and the recent one by Waleed Noor \cite{WaleedNoor}. In this framework, the classical theorem of Beurling \cite{Beurling1955} states  that the property that the region
\[\Omega_t := \{s \in \C: \Re(s) >t\}\]
contains no zeros of $\zeta$ is equivalent to the fact that the constant function $1$ lies in the closure of a particular family of functions in $L^{1/t}(0,1)$, where the interesting case is when $p=1/t \in [1,2]$ (the case $p=2$ was already considered in Nyman's thesis).

\medskip

More concretely, if we define for $k \geq 2, n \in \N$, the functions
\begin{equation}\label{rk}
r_k(n) = k \left\{\frac{n}{k}\right\},
\end{equation}
where $\{ \cdot \}$ denotes the fractional part,  B\'aez-Duarte \cite{BaezDuarte, BaDu1} showed that Beurling's version of the problem could be translated into the problem of approximating arbitrarily close the function $\frac{1}{1-z}$ in the \emph{weighted Dirichlet space} $D_{-2}$ by linear combinations of the family of functions
\[R_k (z) = \sum_{n=1}^\infty r_k(n) z^{n-1},\]
with $k\geq 2$. Recall that the \emph{weighted Dirichlet space} $D_\alpha$, where $\alpha \in \R$, consists of the holomorphic functions $f(z) = \sum_{n=0}^\infty a_n z^n$ in the unit disc $\D$  such that
$$ \|f\|^2_{\alpha}:= \sum_{n=0}^\infty |a_n|^2 (n+1)^{\alpha} < \infty.$$

\medskip

Weighted Dirichlet spaces are particular instances of the so-called  weighted Hardy spaces $H^2(\beta)$ associated to a sequence $\beta=\{\beta_n\}_{n\geq 0}$ introduced by Shields in the seventies in order to study weighted shifts operators \cite{Shields}. They are Hilbert spaces of analytic functions and play an important role in Operator Theory. Particular instances of $\alpha$'s yield classical spaces of analytic functions, namely, for $\alpha=-1$, $D_\alpha$ is the classical Bergman space $A^2$, $\alpha=0$ yields the Hardy space $H^2$ and $\alpha=1$, the Dirichlet space $D$. Note that the continuous inclusion $D_{\alpha_1} \subsetneq D_{\alpha_2 }$ holds for all $\alpha_2 <\alpha_1$. Moreover, when $\alpha >1$ the spaces $D_{\alpha} $ are continuously embedded in the disc algebra $\mathcal{A}$.
We refer to \cite{DS, Duren, EFKMR} for more on these spaces, and to \cite{BaDu1} for the connection mentioned here with $\zeta$.
In our work, a relevant role will be played by $D_\alpha$ when $\alpha \in [-3,-2]$. Broadly speaking, the relevance of translating an approximation problem in a Banach space into an equivalent problem in a Hilbert space radicates in the fact that in Hilbert spaces there are nicer algorithms for optimization, orthogonal projections and other tools based on orthogonality. This allows for all finite dimensional truncations of the problem to be efficiently solved. In this way, problems become more constructive.

\medskip

In 2004, Balazard and Saias \cite{BalazardSaias4} showed that the B\'aez-Duarte criterion was connected with estimating sums defined in terms of the M\"obius function $\mu$.
A (previously known) key fact in their work was that, if $1/(1-z)$ is approximated by $\frac{1}{z}\sum_{k=2}^n c_{k,n} R_k$, then for each $k \geq 2$, the limit as $n$ tends to $\infty$ of $c_{k,n}$ must be a specific value, namely,
\[\lim_{n\to \infty} c_{k,n}= c_{k, \infty} = - \mu(k)/k.\]
In \cite{Wei07}, Weingartner answered a question in \cite{BalazardSaias4} and showed, in particular, that if $c_{k,n}$ are chosen as the coefficients of the orthogonal projections onto the natural finite spaces spanned by $\{R_k\}_{k=2}^n$, then RH is equivalent to the above formula.
\medskip

More recently, Waleed Noor \cite{WaleedNoor} has reformulated further the problem by means of isometries turning B\'aez-Duarte criterion into a question about approximating the constant function $1$ in the Hardy space $H^2$ by means of linear combinations of the sequence of functions $\{h_k\}_{k = 2}^\infty$ given by
\begin{equation}\label{eqhk}h_k(z)= \frac{1}{1-z} \log \left(\frac{1+z+...+z^{k-1}}{k}\right), \quad (z \in \D).\end{equation}

\medskip

In this setting, our main contribution is providing conditions on sums related to the M\"obius function that guarantee the lack of zeros of the Riemann zeta function $\zeta$ in the regions $\Omega_t$.  Our approach is based on finding new ways of expressing the Beurling criterion in other spaces of analytic functions following the spirit of B\'aez-Duarte. Nevertheless, the criteria we exhibit are of a radically different nature to other known criteria associated to the M\"obius function and, in particular, it will allow us to provide a new condition equivalent to the Prime Number Theorem and establish, likewise, an extension of Waleed Noor's Theorem to the setting of weighted Dirichlet spaces.

\smallskip

The rest of the manuscript is organized as follows. We close this introductory section with some preliminaries regarding the spaces of analytic functions in the disc  which will play a key role in order to provide the desired estimates. In Section \ref{section 2}, we prove Theorem \ref{thmA} regarding zero-free regions of the Riemann zeta function and convergence in weighted Dirichlet spaces. If one of the approximation problems in this Theorem holds, the Prime Number Theorem follows (see Subsection \ref{subsection 2}). In Section \ref{section 3} we establish a sufficient condition on the behavior of sums related to the M\"obius function guaranteeing that the Riemann zeta function has zero-free regions within the critical strip (Theorem \ref{thm-3.1}). A key tool in this context is a classical theorem of S. Selberg \cite{Selberg}. In Section \ref{section 4} we consider different approximations of the constant sequence $\mathds{1}$ in $\Span \{r_k: \, k\geq 2\}$ \emph{\`a la} Balazard and Saias. Finally, in Section \ref{section 5}, we present a generalization of Waleed Noor's techniques, by introducing a family of conditions, each of which guarantees a zero-free region for the Riemann $\zeta$ function.

\subsection{Preliminaries}\label{preliminares}
Throughout the rest of the manuscript, $\D$ will denote the unit disc of the complex plane $\C$. The following weighted $\ell^p$ spaces of analytic functions on $\D$ will be essential regarding estimates of sums related to the M\"obius function.

\begin{defn} Let $1 \leq p \leq 2$ and $\alpha \in \R$ be fixed. The space $X^p_{\alpha}$ consists of holomorphic functions $g(z) = \sum_{k=0}^{\infty} a_k z^k$ in $\D$ such that
the norm
$$\|g\|_{p,\alpha}:=\left ( \sum_{k=0}^{\infty} |a_k|^p (k+1)^{\alpha} \right )^{1/p}$$
is finite.
\end{defn}

Note that $X^p_{\alpha}$ are Banach spaces of analytic functions which comprise the introduced weighted Dirichlet spaces when $p=2$, namely, $D_\alpha=X^2_\alpha$, or the Wiener Algebra $\mathcal{W}$, that is, $\mathcal{W}=X^1_0$. In particular, $X^p_\alpha$ are examples of \emph{functional Banach spaces} (see \cite[Definition 1.1]{CowenMcCluer}).

\smallskip

To each function $g$ in $X^{p}_\alpha$ we associate the sequence of Taylor coefficients of $zg(z)$, and define the sequence space $\ell^p_\alpha$   for $1 \leq p \leq 2$ and $\alpha \in \R$ as
\begin{equation}\label{definicion lp}
\ell^p_\alpha :=\{ u=\{u(k)\}_{k\geq1}:\,  g(z):=\sum_{k=1}^\infty u(k) z^{k-1} \in X^{p}_\alpha\}.
\end{equation}
We endow $\ell^p_\alpha$ with the $X^p_{\alpha}$-norm:
$$\|\{u(k)\}_{k\geq 1}\|_{p,\alpha}= \left \| \sum_{k=1}^\infty u(k) z^{k-1} \right \|_{p,\alpha},$$
which turns $\ell^p_\alpha$ into a Banach space. Denoting the sequence with constant value 1 by $\1$, that is, $\1(n)=1$ for all $n\geq 1$, it is clear that $\1 \in \ell^p_\alpha$ for all $\alpha < -1$. Indeed, $\1$ corresponds to the function $\frac{1}{1-z} = \sum_{k=0}^{\infty} z^k$.

\section{Zero-free regions of the Riemann zeta function} \label{section 2}

For $k \geq 2$, recall the $r_k$ function defined on $\N$ introduced in \eqref{rk} by
\begin{equation*}
r_k(n) = k \left\{\frac{n}{k}\right\},\quad (n \in \N).
\end{equation*}
The main goal of this section is proving the following theorem:

\begin{thm}\label{thmA}
Let $1 < p \leq 2$ and $\alpha \in (-1-p,-1-\frac{p}{2}]$. Assume that there exists a sequence of linear combinations $\{S_N=\1+ \sum_{k=2}^N c_{k,N}\, r_k\}_{N=1}^{\infty}$ converging to $0$ in the $X^p_\alpha$-norm.
Then the Riemann zeta function $\zeta$ has no zeros in the half-plane $\Omega_{p,\alpha}:=\{s\in \C:\; \Re(s) > -\frac{\alpha+1}{p}\}$.
Moreover, if the same holds for $p=1$ and $\alpha \in [-2,-3/2]$, then $\zeta$ has no zeros lying on the \emph{closed} half-plane $\Omega_{1,\alpha}:=\{s \in \C:\; \Re(s) \geq -(\alpha+1)\}$.
\end{thm}

\medskip

\begin{rem}
In particular, the previous theorem addresses a way to prove the Prime Number Theorem: we will see it is enough to show the convergence condition for the case $p=1$, $\alpha=-2$. On the other hand, the case $p=1$, $\alpha = -3/2$ cannot happen: if there was such a convergence condition there would be no non-trivial zeros of the zeta function at all, which we know to be false.
\end{rem}

\medskip

Before proving Theorem \eqref{thmA}, we observe the following:
\medskip

Let $m \in \N$ be fixed. From the work in \cite{BaDu1}, a natural candidate for approximating $\mathds{1}$ in $\Span \{r_k: \, k\geq 2\}$ in $X^2_{-2}$ norm appears to be
$$-\sum_{k=2}^m r_k \frac{\mu(k)}{k}.$$
By modifying this approach slightly, we study the following attempt at approximating $\mathds{1}$:
\begin{equation}\label{eqn1100}
F_m := \mathds{1} +  \left( \sum_{k=2}^m r_k \frac{\mu(k)}{k} - \sum_{k=1}^m r_m \frac{\mu(k)}{k} \right).
\end{equation}
A similar result based on the same approximation was already done in \cite{BaDu1} but here we get two advantages: first, we arrive to \emph{closed} half-planes in some parameter cases; secondly, we do obtain a two parameter family of spaces providing tools for the lack of zeros of $\zeta$. This may allow future research on the necessary estimates to relax or strengthen, providing additional information. 

\smallskip

Accordingly, in order to apply this result, we will make appropriate choices of $m$'s belonging to an infinite subset $M \subset \N$ of values to obtain accurate estimates in the required approximation. Indeed, as we will show, in order to provide a new proof of the Prime Number Theorem, we may take $M=\N$. Nevertheless, for more advanced estimates,
it will be necessary to determine a suitable choice of the subset $M$.

\medskip

In order to prove Theorem \ref{thmA}, we introduce a set of real functions defined on $(0,1)$ by
\begin{equation}\label{set C}
\mathcal{C}=\left\{ f(x) = \sum_{k=1}^N c_k \left\{\frac{1}{kx}\right\}:\, x \in (0,1), \;   \sum_{k=1}^N \frac{c_k}{k} =0 \right\}.
\end{equation}
Note that $\mathcal{C}\subset L^p(0,1)$ for $1\leq p\leq \infty$. Moreover, each function in $\mathcal{C}$ is constant on intervals of the form $\left(\frac{1}{n+1},\frac{1}{n}\right]$ for $n \geq 1$: the value of $c_1$ is determined by the remaining values so that a basis of $\mathcal{C}$ is given by the functions of the form $g_k(x)= k \left\{ \frac{1}{kx}\right\} - \left\{ \frac{1}{x}\right\}$,  for $k \geq 2$. On the interval $\left(\frac{1}{n+1},\frac{1}{n}\right]$, these functions have constant values $g_k(x) = r_k(n)$.

We will also need the following lemma auxiliary to the main result of Beurling in \cite{Beurling1955} (see formula (2) there):

\begin{lem}\label{lem1}
If $f \in \mathcal{C}$ with $f(x) =\sum_{k=1}^N c_k \left\{\frac{1}{kx}\right\}$, $x\in (0,1)$ and $\Re (s) >0$ then
\[\int_0^1 f(x) x^{s-1} dx= -\frac{\zeta (s) \sum_{k=2}^N c_k k^{-s}}{s}.\]
\end{lem}

\smallskip

From Lemma \ref{lem1}, adding $1/s$ on both sides makes the following expression valid on any $s$ with $\Re (s) >0$:
\begin{equation}\label{lem1 rev}
\int_0^1 (1+f(x)) x^{s-1} dx= \frac{1}{s}\left(1- \zeta (s)  \sum_{k=2}^N c_k k^{-s}\right) .
\end{equation}

\smallskip

Suppose, for the moment, we can find functions in $\mathcal{C}$, namely $\{f_N\}_{N=1}^{\infty}$, making the absolute value of the integral in the left-hand side arbitrarily small, and that $s$ is a zero of $\zeta$ with $\Re (s) >0$. Then it must happen that $0=\frac{1}{s},$ which is a contradiction.

\medskip

This is to be read as a possibility for showing that a fixed value $s$ is not a zero of $\zeta$: we only need to show that the integral on the left hand side of \eqref{lem1 rev} can be made arbitrarily small (by choosing $f \in \mathcal{C}$). Therefore our aim in order to prove Theorem \ref{thmA}  is to find a sequence converging to the function $1$ by means of functions $\{f_N\}_{N=2}^{\infty}$ in terms of the integral.

\medskip

\begin{demode}\emph{Theorem \ref{thmA}.}
Note that approximating the function $1$ with linear combinations of functions of the form $\left\{\frac{1}{kx}\right\}$ is equivalent to approximating the sequence of its values over intervals of the form $\left(\frac{1}{n+1},\frac{1}{n}\right)$, by the locally constant values of functions in $\mathcal{C}$. Hence, if $f \in \mathcal{C}$ with $f(x) =\sum_{k=1}^N c_k \left\{\frac{1}{kx}\right\}$, $x\in (0,1)$, we deduce that for complex numbers $s$ with $\Re (s)>0$,

\begin{eqnarray}\label{eqn1000}
\int_0^1 (1+f(x)) x^{s-1} dx &=& \sum_{n=1}^{\infty} \int_{\frac{1}{n+1}}^{\frac{1}{n}} (1+f(x)) x^{s-1} dx \nonumber \\
&=& \displaystyle \sum_{n=1}^{\infty} \left(1+f\left(\frac{1}{n}\right)\right) \int_{\frac{1}{n+1}}^{\frac{1}{n}}  x^{s-1} dx  \nonumber\\
&=& \displaystyle \sum_{n=1}^{\infty} \left( 1+\sum_{k=2}^N c_k r_k(n) \right) \int_{\frac{1}{n+1}}^{\frac{1}{n}}  x^{s-1} dx.
\end{eqnarray}
In the last step here, we used that $r_k(n) = k \left\{\frac{n}{k}\right\} = k \left\{\frac{1}{kx}\right\}- \left\{\frac{1}{x}\right\}$ for all $x \in (\frac{1}{n+1}, \frac{1}{n}]$, and that $f \in \mathcal{C}$ and thus $c_1= -\sum_{k=2}^N \frac{c_k}{k}$.

\medskip

Taking absolute values and writing the last expression in terms of the sequence $S_N$ whose existence is guaranteed by hypothesis, \eqref{eqn1000} turns into
\begin{equation}\label{eqn1200}
\left|\sum_{n=1}^{\infty} S_N(n) \int_{\frac{1}{n+1}}^{\frac{1}{n}}  x^{s-1} dx\right| \leq \frac{1}{\Re s} \sum_{n=1}^{\infty} \left|S_N(n)\right|  \left|\frac{1}{n^{\Re(s)}}-\frac{1}{(n+1)^{\Re(s)}}\right| .
\end{equation}

\medskip

Assume first that $p=1$, $\alpha \in [-2, -3/2]$, and $\Re(s) \geq -(\alpha+1)$. Then the right hand side \eqref{eqn1200} is bounded by  $2 \sum_{n=1}^{\infty} \left|S_N(n)\right|(n+1)^{\alpha}$ which converges to $0$ by assumption.

\medskip

On the other hand, if $1<p \leq 2$ and $\alpha \in (-1-p,-1-\frac{p}{2}]$, bearing in mind that $n^{-\Re(s)} - (n+1)^{-\Re(s)} \approx n^{-1-\Re(s)}$, turns \eqref{eqn1200} into
\begin{eqnarray}\label{eqn1200-1}
\left| \sum_{n=1}^{\infty} S_N(n)  \int_{\frac{1}{n+1}}^{\frac{1}{n}}  x^{s-1} dx\right| & \leq & \sum_{n=1}^{\infty}   \left|S_N(n)\right| n^{-1-\Re(s)} \nonumber \\
& \leq &   \left(\sum_{n=1}^{\infty} \left|S_N(n) \right|^p n^{\alpha}\right)^{1/p}  \left(\sum_{n=1}^{\infty} n^{-q(1+\Re(s)+\frac{\alpha}{p})}\right)^{1/q}
\end{eqnarray}
where last inequality follows upon applying H\"older inequality. Here $q$ is the reciprocal of $p$, namely $\frac{1}{p}+\frac{1}{q}=1$.

The first factor on the right hand side of \eqref{eqn1200-1} converges to 0 by hypothesis while the second factor is a convergent sum because the exponent of $n$ in the sum is smaller than $-1$, which is true precisely because $\Re(s) > -\frac{\alpha+1}{p}$. This concludes the proof of Theorem \ref{thmA}.
\end{demode}

\begin{rem}
Applying H\"older inequality to the absolute value of the left-hand side is at the core of Beurling's criterion for establishing regions free of $\zeta$ zeros.
In obtaining the B\'aez-Duarte criterion, this last step is done \emph{after} applying H\"older inequality to the integral in the spirit of Beurling, but here we decidedly want to perform this step prior to the introduction of estimates.
\end{rem}

\subsection{An application of Theorem \ref{thmA} connected to the Prime Number Theorem}\label{subsection 2}
Our aim in this subsection is to establish a first connection to the famous Prime Number Theorem: \emph{If $\pi(x)$ is the number of primes less than or equal to $x$, then
$x^{-1} \pi (x) \log x\to 1$  as $x\to \infty$}. That is, $\pi(x)$ is asymptotically equal to $x/\log x$ as $x\to \infty$.
It is well-known that the Prime Number Theorem can be derived from the following statement:

\medskip

\begin{PNT}
Let $s\in \C\setminus \{1\}$ with $\Re(s) = 1$. Then $\zeta(s) \neq 0$.
\end{PNT}

\medskip

With the definition of $F_m$ given in \eqref{eqn1100} and Theorem \ref{thmA} at hand,  the Prime Number Theorem will be connected with the case $p=1$ of the following result. Here the proof works for all positive values of $p$ even though for $p<1$ the spaces are not Banach and some of the previous results may then fail.

\begin{prop}\label{lemC}
Let $0 < p < \infty$. If $\alpha = -p-1$, then there exists a constant $C$ such that for all $m \in \N$,
\[\|F_m\|_{p,\alpha} \leq C.\]
If $\alpha <-p-1$, then \[\|F_m\|_{p,\alpha} \rightarrow 0 \quad \text{ as } \quad m \rightarrow \infty.\]

\end{prop}

\medskip

Suppose, for the moment,  that Proposition \ref{lemC} is proved. What is the connection with the Prime Number Theorem?

Let $p=2$ and $\alpha=-3$ in Proposition \ref{lemC}. Then $\{F_m\}_{m\geq 1}$ is a bounded sequence in $X^2_{-3}$ norm. In addition, for $\alpha < -3$ the sequence of functions $\{F_{m_k}\}_{k \geq 1}$, where  $m_k = 2^{2^k}$, converges to zero in $V:=X^2_\alpha$-norm. Accordingly, $\{F_{m_k}\}_{k \geq 1}$ converges pointwise to zero. Now, $X^2_{-3}$
is a functional Banach space such that the span of the point evaluation linear functionals are dense in its dual space. Hence,
$\{F_{m_k}\}_{k \geq 1}$ is a weakly convergent sequence (to zero) in $X^2_{-3}$ (see \cite[Proposition 1.2]{CowenMcCluer}). Now, Mazur's Lemma (see \cite[Theorem 1, Chapter II]{Diestel}, for instance) yields the existence of a sequence of convex combinations $\{S_N\}_{N \geq 1}$ that is actually convergent to zero in the space norm.

Unfortunately, this argument does not work with the extreme case $p=1$, but we can at least derive the following statement directly from Theorem \ref{thmA}:
\begin{corollary}
Suppose that $\{F_m\}_{m\in\N}$ satisfy $\|F_m\|_{1,-2}$ tends to $0$ as $m$ tends to $\infty$. Then, the Prime Number Theorem follows.
\end{corollary}
We will later see that from a well known equivalent formulation of the Prime Number Theorem, the assumption in this hypothesis can be shown.

\medskip

At this point, it is worthy to note that those $\alpha$'s satisfying the last statement of Proposition \ref{lemC} must have a maximum threshold. From applying Theorem \ref{thmA} to the sequence $F_m$, larger zero-free regions of the Riemann zeta function would be deduced. Indeed, if such an estimate was available for $\|F_m\|_{1,-3/2}$ or $\|F_m\|_{p, c}$ for $p \in (1,2]$ and $c > -1-\frac{p}{2}$, it would mean the Riemann zeta function would have no non-trivial zeros at all, but the existence of some such zeros were even known to Riemann.

\medskip

We are left with the task of proving Proposition \ref{lemC} and for this end, we require to show some properties of the approximation errors $F_m$.

\subsubsection{Basic properties of $F_m$}

Let $\mu(k)$ be the M\"obius function evaluated at $k \in \N$. It is standard that \[\sum_{k=1}^n \mu(k) \left\lfloor \frac{n}{k} \right\rfloor = 1,\]
where $\lfloor x \rfloor$ denotes the integer part of $x \in \R$.

Likewise, the property that, as $m$ grows to $\infty$, \[\left|\sum_{k=1}^m \frac{\mu(k)}{k} \right| = o(1),\]  can be proved to be equivalent to the Prime Number Theorem. This sum is in any case always bounded by 1. See Lemma 3.2, Remark 3.3, and Theorems 5.9 and 5.11 in \cite{BateDia}.
If $n,m \in \N$, we define the 2-variable functions
\[G(n,m)=\sum_{k=1}^m \mu(k) \left\lfloor \frac{n}{k} \right\rfloor,\]
and
\[H(n,m)= \sum_{k=2}^m \mu(k) \left\{ \frac{n}{k} \right\}.\]

Having in mind the expression \eqref{eqn1100} for $F_m$, we notice the term on $k=m$ is $0$. In addition, $F_m$ can be easily evaluated at any $n \in \N$:
\begin{equation}\label{eqn11}
F_m (n) = 1 + H(n,m) - r_m(n) \sum_{k=1}^m  \frac{\mu(k)}{k},
\end{equation}

\noindent which can be simplified by decomposing $r_t(q) = t\{q/t\} = q - t \lfloor q/t \rfloor$, giving

\begin{align*}
F_m (n) &= 1 + n \sum_{k=2}^m \frac{\mu(k)}{k} - \sum_{k=2}^m  \mu(k) \left \lfloor \frac{n}{k} \right \rfloor- n \sum_{k=1}^m \frac{\mu(k)}{k} + m \left \lfloor \frac{n}{m}\right \rfloor  \sum_{k=1}^m   \frac{\mu(k)}{k} \\
&= 1 - \sum_{k=2}^m  \mu(k) \left \lfloor \frac{n}{k} \right \rfloor- n + m \left \lfloor \frac{n}{m}\right \rfloor  \sum_{k=1}^m   \frac{\mu(k)}{k}. \\
\end{align*}
Hence,
\begin{equation}\label{eqn13}
F_m(n) = 1 - G(n,m) + m \left \lfloor \frac{n}{m} \right \rfloor  \sum_{k=1}^m   \frac{\mu(k)}{k}.
\end{equation}

\noindent In particular, when $n < m$, then $\left \lfloor \frac{n}{m} \right \rfloor =0 $ and therefore
$$\sum_{k=1}^m \mu(k) \left \lfloor \frac{n}{k} \right \rfloor = \sum_{k=1}^n \mu(k) \left \lfloor \frac{n}{k} \right \rfloor =1.$$
This means that
\begin{equation}\label{eqn10}
F_m(n) = 0, \quad n \leq m.
\end{equation}

\noindent Likewise, from \eqref{eqn11} it follows that
\begin{equation}\label{eqn12}
|F_m(n)| \leq 1 + \sum_{k=2}^m 1 + r_m(n) < 2m, \quad n,m \in \N.
\end{equation}
\medskip

\noindent With \eqref{eqn10} and \eqref{eqn12} at hand, the proof of Proposition \ref{lemC} goes as follows.

\medskip

\medskip

\begin{demode}{\emph{Proposition \ref{lemC}}. } Suppose $\alpha \leq -p-1$, where $p>0$. Upon applying \eqref{eqn10} and  \eqref{eqn12} we deduce
\[
\|F_m\|_{p,\alpha}^p = \sum_{n=m}^\infty |F_m(n)|^p (n+1)^{\alpha} \lesssim \frac{m^p}{|\alpha +1 |} (m+1)^{\alpha+1},
\]
which converges to $0$ if $\alpha <-p-1$ and remains bounded if $\alpha=-p-1$ as claimed.
This completes the proof. \end{demode}

%%%%%%%%%%%%%%%%%%%

\section{M\"obius estimates and zero-free regions of the Riemann zeta function}\label{section 3}

As mentioned in the introduction, the distribution of the zeros of the Riemann zeta function is strongly connected to the behavior of the M\"obius function $\mu$. The main goal of this section is to establish a sufficient condition on the behavior of sums related to the M\"obius function which guarantees that the Riemann zeta function has some zero-free regions within the critical strip, namely, the strip consisting of those complex numbers with imaginary part in $(0,1)$. More precisely, we prove the following:

\begin{theorem}\label{thm-3.1}
Let $\s \in [\frac{1}{2},1)$ and $\omega: [0, +\infty) \rightarrow [0, +\infty]$ be a non-decreasing function satisfying $\omega\left(\frac{1}{2}\right) < \infty$.
Suppose that $m \in \N$ and for $n \in (m, m^{1/\s})$, we have
\begin{equation}\label{eqn6}|G(n,m)| \leq \left(m^{\s} + \frac{n^{\s}}{\sqrt{\log(n)}} + \left(\frac{n}{m}\right)^{\frac{\s}{1-\s}} \right) \cdot \omega\left(m \left|\sum_{k=1}^m \frac{\mu(k)}{k}\right|\right).
\end{equation}
Then the Riemann zeta function has no zeros in the region
\[\Omega_{\s} = \{s \in \C: \Re(s) > \s\}.\]
\end{theorem}

\smallskip

Before proceeding with the proof, a few remarks are in order.
\smallskip

\begin{rem}$\;$ \newline
\begin{enumerate}
\item Instead of assuming \eqref{eqn6}, the same result follows if $G$ is replaced by $H$: Since $\frac{\s}{1-\s} \geq 1$, and $\left\lfloor \frac{n}{k} \right\rfloor = \frac{n}{k} - \left\{ \frac{n}{k} \right\}$ the change of $G$ by $H$ only affects the choice of function $\omega(t)$ by adding (at most) a term $t+1$ to it.

\item For $H$, the hypotheses are trivial in the limit $\s=1 \equiv \omega$ since $|H(n,m)| \leq m$, even changing the term $\left(\frac{n}{m}\right)^{\frac{\s}{1-\s}}$ by the smaller $\frac{n}{m}$. In that case, for $G$, \eqref{eqn6} also holds in the limit $\s=1$ with $\omega(t)=1+t$.

\item When $n \not  \in (m, m^{1/\s})$, \eqref{eqn6} does hold automatically with $\omega(t)=2+t$. Indeed:
\begin{enumerate}
\item[(c1)] If $n \geq m^{1/\s}$, then \[|H(n,m)| \leq m = m^{(\frac{1}{\s}-1)(\frac{\s}{1-\s})}= \left(\frac{m^{1/\s}}{m}\right)^{\frac{\s}{1-\s}} \leq \left(\frac{n}{m}\right)^{\frac{\s}{1-\s}}.\]
\item[(c2)] If $n \leq m$, then $n/k < 1$ for $n < k \leq m$ and thus \[|G(n,m)| = |G(n,n)|=1.\]
\end{enumerate}

\item For $\s=1/2$ and $\omega(t)=1+t$, the assumption \eqref{eqn6} does hold at least for $n \leq 360.000$. At the same time, for $\s= 1/2$, the celebrated work of Odlyzko and te Riele regarding the Mertens function \cite{OtR} points either towards $\omega$ being a function larger than 1 for the result to apply, or towards the term $m^{\s}$ being dominated by the terms involving $n$.

\item When $\left|\sum_{k=1}^m \frac{\mu(k)}{k} \right| \geq \frac{1}{\sqrt{m}}$, there is nothing to be shown. This means our condition goes in a direction that is completely different from other known conditions for Riemann Hypothesis based on $\mu$: The values of $m$ for which $\left|\sum_{k=1}^m \frac{\mu(k)}{k}\right| \geq 1/\sqrt{m}$ are not an obstruction at all. In particular, by S. Selberg's theorem \cite{Selberg}, there are infinitely many values of $m \in \N$ such that
\[m\left|\sum_{k=1}^m \frac{\mu(k)}{k} \right| \leq \frac{1}{2}.\]
The proof of Theorem \ref{thm-3.1} is based on the existence of these values.

\item Depending on the value of $n$,  the leading term of the first factor of the right-hand side of \eqref{eqn6} is,
\begin{align*}
m^{\s}, \quad &\text{ when } \quad n \in \left(m, m (\log m)^{\frac{1}{2{\s}}}\right],\\
\frac{n^{\s}}{\sqrt{\log(n)}}, \quad &\text{ when } \quad n \in \left(m (\log m)^{\frac{1}{2{\s}}}, \frac{m^{1/{\s}}}{(\log m)^{\frac{1-{\s}}{2{\s}^2}}}\right],\\
\left(\frac{n}{m}\right)^{\frac{\s}{1-\s}}, \quad &\text{ when } \quad n \in \left(\frac{m^{1/{\s}}}{(\log m)^{\frac{1-{\s}}{2{\s}^2}}}, m^{1/{\s}}\right].
\end{align*}

\end{enumerate}
\end{rem}
%%%%%%%%%%%%%%%%%%%%%%%%%%%%%%%%%%%%%%%%%%%%
\begin{proof}

Let us apply Theorem \ref{thmA} with $p=2$ and $\alpha = -(1+2{\s})$. In Hilbert spaces, we can apply Mazur's Theorem  which tells us that the weak closure of a linear space is equal to its closure. Thus, we just need to prove that under our hypotheses $F_m$ is weakly convergent in the $X^2_\alpha$ norm (or equivalently, in the $D_\alpha$ norm). Since we know that pointwise convergence holds from Proposition \ref{lemC}, then we can procede as we did when mentioning the Prime Number Theorem: it is enough to check that $\|F_m\|_{p,\alpha}$ stays bounded when $m\in M$ for some (increasing)  subsequence $M=\{m_r\}_{r\in\N} \subset \N$.

In order to evaluate the $X^2_{\alpha}$-norm of $F_m$,  firstly, we apply \eqref{eqn10} to get rid of terms where $n \leq m$. Then, by means of the estimate \eqref{eqn12} applied to the tail of the sum, we obtain:
\begin{align*}
\|F_m\|^2_{2,{\s}} &= \sum_{n = m}^{\lfloor m^{1/{\s}} \rfloor} |F_m(n)|^2 \left(\frac{1}{n^{2{\s}}}-\frac{1}{(n+1)^{2{\s}}}\right) \\  &+ \sum_{n = \lfloor m^{1/{\s}}\rfloor +1}^{\infty}  |F_m(n)|^2 \left(\frac{1}{n^{2{\s}}}-\frac{1}{(n+1)^{2{\s}}}\right) \\
& \leq \sum_{n = m}^{\lfloor m^{1/{\s}} \rfloor} |F_m(n)|^2 \left(\frac{1}{n^{2{\s}}}-\frac{1}{(n+1)^{2{\s}}}\right) + 4 m^2 m^{-2}.
\end{align*}

The only thing remaining is to show that under the hypotheses of Theorem \ref{thm-3.1}, there is a $C>0$ and an (increasing) subsequence $M=\{m_j\}_{j\in\N} \subset \N$ such that for $m \in M$,
\begin{equation}\label{eqn17}
\sum_{n = m}^{\lfloor m^{1/{\s}} \rfloor} |F_m(n)|^2 \left(\frac{1}{n^{2{\s}}}-\frac{1}{(n+1)^{2{\s}}}\right) \leq C.
\end{equation}

By a theorem of S. Selberg \cite{Selberg}, there are infinitely many values of $m \in \N$ for which the quantity $\sum_{k=1}^m \frac{\mu(k)}{k}$ changes sign. If $m$ is such that \[\left( \sum_{k=1}^{m-1} \frac{\mu(k)}{k}\right) \cdot \left( \sum_{k=1}^m \frac{\mu(k)}{k} \right) < 0,\] then we have both
\[\left| \sum_{k=1}^{m-1} \frac{\mu(k)}{k} \right| \leq \frac{1}{m} \quad \text{ and } \quad  \left| \sum_{k=1}^m \frac{\mu(k)}{k} \right| \leq \frac{1}{m}.\]
In fact, we can then choose either $m$ or $m-1$ and there are infinitely many values of $m \in \N$ such that
\begin{equation}\label{eqn20}
m \left| \sum_{k=1}^m \frac{\mu(k)}{k} \right| \leq \frac{1}{2}.
\end{equation}
We restrict to those values of $m$, i.e., $M$ is formed by the values $m \in \N$ with the property \eqref{eqn20} and from now on we assume that $m \in M$. Bearing that in mind, we will denote $t = m \left| \sum_{k=1}^m \frac{\mu(k)}{k} \right|$.

Then \eqref{eqn13} leads to
\[|F_m(n)| \leq 1 + t (n/m) + \left|\sum_{k=1}^{m} \mu(k) \lfloor \frac{n}{k} \rfloor \right|.\]
Notice that $1 \leq \frac{{\s}}{1-{\s}}$ and therefore, if \eqref{eqn6} holds, $F_m$ satisfies
\[|F_m(n)| \leq \left( m^{\s} + \frac{n^{\s}}{\sqrt{\log n}} + \left(\frac{n}{m}\right)^{\frac{\s}{1-\s}}\right) \omega'(t),\]
where $\omega'(\s) := \omega(\s) + 1+ \s$.

Let us finally check that \eqref{eqn17} holds. By our last considerations, $F_m$ can be controlled with $\omega'$ yielding
\begin{align*}
\sum_{n=m}^{{\lfloor m^{1/{\s}} \rfloor}} &|F_m(n)|^2 \left( \frac{1}{n^{2\s}} - \frac{1}{(n+1)^{2\s}}\right) \\
&\lesssim (\omega'(t))^2  \cdot \left( \sum_{n=m}^{{\lfloor m^{1/{\s}} \rfloor}} \left( m^{2\s} + \frac{n^{2\s}}{\log n} + \left(\frac{n}{m}\right)^{\frac{2\s}{1-\s}}\right) \cdot \left( \frac{1}{n^{2\s}} - \frac{1}{(n+1)^{2\s}}\right)  \right)\\
&\lesssim \omega'(1/2)^2  \cdot \left( 1 + \sum_{n=m}^{{\lfloor m^{1/{\s}} \rfloor}} \frac{1}{n \log n} + m^{\frac{-2\s}{1-\s}} \cdot \left( \sum_{n=m}^{{\lfloor m^{1/{\s}} \rfloor}} n^{\frac{2\s}{1-\s}-1-2\s}\right) \right) \\
&:= \omega'(1/2)^2 \cdot \left( 1+ \Sum_1 +  \Sum_2 \right).
\end{align*}
where we used that $\left( \frac{1}{n^{2\s}} - \frac{1}{(n+1)^{2\s}}\right) \approx n^{-1-2\s}$.

Since $2\s ( \frac{1}{1-\s} -1) -1 > -1$, we have a bound for $\Sum_2$:
\[\Sum_2 \lesssim m^{\frac{-2\s}{1-\s} + \frac{1}{\s} \cdot \frac{2\s^2}{1-\s}} = 1.\]

On the other hand,
\[\Sum_1 \lesssim [\log \log x]_{m}^{m^{1/\s}} = \log (1/\s)\]
which is a constant. This concludes the proof.
\end{proof}

\begin{rem}
A similar proof is possible, based on a $p \neq 2$ version of the previous, but unfortunately, a different value of $p$ would have required that the denominator in the term $\frac{n^{\s}}{\sqrt{\log n}}$ from \eqref{eqn6} be replaced with $\frac{n^{\s}}{(\log n)^{1/p}},$ which is a stronger assumption for $p<2$.
\end{rem}

\medskip

\subsection{Further remarks on Theorem \ref{thm-3.1}: towards the needed estimates on M\"obius function sums}
$\;$\newline
Clearly, assumption \eqref{eqn6} plays a key role in the proof of Theorem \ref{thm-3.1}. In order to provide some insights into \eqref{eqn6}, we show estimates which involve the classical \emph{Mertens function}.

Recall that the Mertens function $M:\N \rightarrow \N$ is defined by
\[M(\imath) = \sum_{n=1}^\imath \mu(n),\]
and many authors have studied possible relations with the Riemann zeta function, some of which are present in \cite{BateDia}. The following is already included in \cite{BaDu1}:

\medskip

\begin{prop} For $n, m \in \N$ we have
\[G(n,m) = 1+ \sum_{\imath=1}^{ \left \lfloor \frac{n}{m} \right \rfloor} \left( M(m) -M(n/\imath)\right).\]
\end{prop}

\begin{proof}
Note that
\begin{align*}
\sum_{k=1}^m \mu(k) \left \lfloor \frac{n}{k} \right \rfloor &= \sum_{\imath = \left \lfloor \frac{n}{m} \right \rfloor}^n \imath \sum_{\stackrel{1 \leq k \leq m}{\left \lfloor \frac{n}{k} \right \rfloor = \imath}} \mu(k) \\
&= \left \lfloor \frac{n}{m} \right \rfloor \left(M(m) - M(\frac{n}{\left \lfloor \frac{n}{m} \right \rfloor +1})\right) + \sum_{\imath = \left \lfloor \frac{n}{m} \right \rfloor + 1}^n \imath \left( M(\frac{n}{\imath}) - M(\frac{n}{\imath+1})\right) \\
&= \left \lfloor \frac{n}{m} \right \rfloor M(m) + \sum_{\imath = \left \lfloor \frac{n}{m} \right \rfloor + 1}^n  M(\frac{n}{\imath}).
\end{align*}
Then, using the fact that we just proved for $m=n$ we have
\[1 = \sum_{k=1}^n \mu(k) \left \lfloor \frac{n}{k} \right \rfloor = \sum_{\imath=1}^n M(n/\imath).\]
This can be substituted above to give

\[\sum_{k=1}^m \mu(k) \left \lfloor \frac{n}{k} \right \rfloor =  \left \lfloor \frac{n}{m}  \right\rfloor M(m) +1 - \sum_{\imath = 1}^{\left \lfloor \frac{n}{m} \right \rfloor}  M(\frac{n}{\imath}) = 1 + \sum_{\imath=1}^{\left \lfloor \frac{n}{m} \right \rfloor} \left(M(m) - M(n/\imath)\right),\]
as claimed.
\end{proof}

We can take advantage of this proposition to show that two values of $n$ that are close to each other satisfy \eqref{eqn6} with similar functions $\omega$, and thus we might only need to prove \eqref{eqn6} for a \emph{dense enough} grid of values $m,n$.

\begin{prop}
Let $m \in \N$, $n=qm+d$ and $n'=qm+d'$ with $d, d' \in \{0,...,m-1\}, q \in \N$. Then
\[|G(n,m)- G(n',m)| \lesssim q + |d-d'| \log r,\]
where $r = e+ \min (|d-d'|, q)$.
\end{prop}

Notice $q \leq n/m$ and so this principle works when $d-d'$ is small compared to the right-hand side in \eqref{eqn6}.

\begin{proof}
First, upon applying the previous Proposition, and having in mind that
\[|M(x)-M(y)| \leq |x-y| +1,\]
we obtain
\begin{align*}
\left|\sum_{k=1}^m \mu(k) \left(\left\lfloor \frac{n}{k} \right\rfloor - \left\lfloor \frac{n'}{k} \right\rfloor\right)\right| &= \left|\sum_{\imath=1}^q \left( M(n/\imath) - M(n'/\imath) \right)\right|\\
\leq \sum_{\imath=1}^q \left( \left\lfloor \frac{|d-d'|}{t} \right\rfloor +1 \right) &\leq q + |d-d'| \sum_{\imath=1}^r \frac{1}{\imath}.
\end{align*}
\end{proof}

%%%%%%%%%%%%%%%%%%%%%%%%%%%%%%%%%%%%%%%%%%%%%
\section{Further approximation of $\mathds{1}$ in $\Span \{r_k: \, k\geq 2\}$ in the $X^2_{-2}$ norm}\label{section 4}

As we discussed in Section \ref{section 2}, the work of Balazard and Saias \cite{BalazardSaias4} suggests natural candidates for approximating $\mathds{1}$ in $\Span \{r_k: \, k\geq 2\}$ in the $X^2_{-2}$ norm, by using analogues of the approximations appearing in \cite{BaDu1}, namely,
$$\sum_{k=2}^m r_k \frac{\mu(k)}{k}.$$
Previously, we modified such approach and considered the sequences $F_m$ given by \eqref{eqn1100}, that is,
$$
F_m= \mathds{1} + \left( \sum_{k=2}^m r_k \frac{\mu(k)}{k} - r_m \sum_{k=1}^m  \frac{\mu(k)}{k} \right),
$$
as an attempt at approximating $\mathds{1}$ in the scheme of Theorem \ref{thmA}.

The aim of this section is approximating $\1$ by a linear combination of a \emph{selection} of $\{r_k: 2 \leq k \leq n\}$ which provides a different insight in order to determine zero-free regions of $\zeta$. In particular, for each $m\geq 2$ we will consider approximations with linear combinations of the $r_k$'s such that $k$ divides $m$, that is, $\Span \{r_k: 2 \leq k  | m\}$.  More precisely, instead of $F_m$ we will consider the sequences

\begin{equation}\label{eqn2100}
F'_m := \mathds{1} + \left( \sum_{\begin{smallmatrix} k=2 \\ k|m\end{smallmatrix}}^m r_k \frac{\mu(k)}{k} - r_m \sum_{\begin{smallmatrix} k=1 \\ k|m\end{smallmatrix}}^m  \frac{\mu(k)}{k} \right).
\end{equation}

In addition, we will restrict ourselves to the subsequence of positive integers $\{m_\jmath\}$  defined as the product of the first $\jmath$ prime numbers, that is, $m_\jmath= \prod_{i=1}^\jmath p_i$ is the prime decomposition of $m_\jmath$. For the sake of simplicity, in what follows we will denote $F'_{m_\jmath}$ by $F'_{m}$.

\smallskip

This choice of approximation also satisfies the assumptions in Theorem \ref{thmA} since, for fixed $n \in \N$ and as $t\rightarrow \infty$, each coefficient $F'_m(n)$ converges towards the corresponding $F_m(n)$. However, as we will show, the conclusions we can obtain are different in nature.  The point of establishing a different selection of the set of $r_k$ here is that the necessary estimates on the M\"obius function sums may have a better chance of being proved, since the arithmetic properties of the M\"obius function with respect to the divisors of a number are well known.

\medskip

A standard property of the M\"obius function that will be useful for us is the following  (see \cite[Theorem 2.16]{BateDia}):
\begin{lem}\label{lemZ}
Let $d \in \N$. Then \[\sum_{\begin{smallmatrix}k=1\\k|d\end{smallmatrix}}^{d} \mu(k) = \delta_d.\]
\end{lem}

\noindent Analogously as in Section  \ref{section 2}, we introduce two double indexed functions
\[G'(n,m)=\sum_{\begin{smallmatrix}k=1\\k|m\end{smallmatrix}}^m \mu(k) \left\lfloor \frac{n}{k} \right\rfloor,\]
\[H'(n,m)= \sum_{\begin{smallmatrix}k=2\\k|m\end{smallmatrix}}^m \mu(k) \left\{ \frac{n}{k} \right\},\]
and deduce:
\begin{equation}\label{eqn2321}
F'_m (n) = 1 + H'(n,m) - r_m(n) \sum_{\begin{smallmatrix}k=1\\k|m\end{smallmatrix}}^m  \frac{\mu(k)}{k},
\end{equation}
and
\begin{equation}\label{eqn2322}
F'_m(n) = 1 - G'(n,m) + m \left\lfloor \frac{n}{m} \right\rfloor  \sum_{\begin{smallmatrix}k=1\\k|m \end{smallmatrix}}^m   \frac{\mu(k)}{k}.
\end{equation}

Consider $\jmath\in \N$ and $m \in \N$, with $m= \prod_{i=1}^\jmath p_i$ both fixed.
Now, from the definition of $F'_m$, and for $n \leq p_\jmath$, notice that $r_k(n)-r_m(n) = 0$ for all $k > p_\jmath$, and thus
\begin{equation}\label{eqn2200}
F'_m(n) =F_{p_\jmath}(n)=0, \quad n \leq p_\jmath.
\end{equation}

Apart from this, we remark a key distinction of $F'_m$ with respect to $F_m$: \emph{$F'_m$ is $m$-periodic} since $r_k$ is $m$-periodic whenever $k |m$.

In $D_\alpha$ spaces for $\alpha \in [-3,-2]$, any $m$-periodic sequence (of Taylor coefficients) defines a function $f$ that has a norm $\|f\|_\alpha$ between $\|q_m(f)\|_\alpha$ and  $2\|q_m(f)\|_\alpha$, where $q_m(f)$ is the Taylor polynomial of $f$ of degree less or equal to $m$. Therefore, in order to show boundedness of $\|F'_m\|_{\alpha}$ we just need to show boundedness of the norm contributed by the coefficients of order $p_\jmath < n< m$ (when $n=m$ then $F'_m(n)=1$). This finite-dimensionality implies, among other things, that all arguments about whether a space is closed are automatically resolved.

\smallskip

By making use of the standard notation $\# E$ for the cardinal of a set $E$ and $(u,v)$ for the greatest common divisor of $u$ and $v$, the first basic result towards understanding the remaining case is the following:

\begin{thm}\label{lemY}
Let $1 \leq n < m$. Then
\[F'_m(n) = -\# \{ 1< k \leq n : (k,m)=1\}.\]
\end{thm}

\begin{proof}
If $n \leq p_\jmath$ then we already know that the result is true from \eqref{eqn2200}. If $n > p_\jmath$, we just need to show that when $(n,m)=1$ we have
\begin{equation}\label{eqn2300}
F'_m(n)=-1+F'_m(n-1),
\end{equation}
whereas otherwise we have
\begin{equation}\label{eqn2301}
F'_m(n)=F'_m(n-1).
\end{equation}
Let us first assume $(n,m)=1$. Then $r_k(n)-r_k(n-1) =1$ for all $k |m$. Thus,
\[F'_m(n)-F'_m(n-1)=\left( \sum_{\begin{smallmatrix} k=2 \\ k|m\end{smallmatrix}}^m \frac{\mu(k)}{k} - \sum_{\begin{smallmatrix} k=1 \\
k|m\end{smallmatrix}}^m \frac{\mu(k)}{k} \right)=-1.\]

On the other hand, if $(n,m)=D \neq 1$, then $r_k(n)-r_k(n-1)=1$ for $k |m$ with $k \nmid n$ while $r_k(n)-r_k(n-1) = 1-k$ when $k | D$.
Hence, the only difference with the previous is contributed by $k$ dividing $D$ and in this case we have
 \[F'_m(n)-F'_m(n-1)=-1- \left( \sum_{\begin{smallmatrix} k=2 \\ k|D\end{smallmatrix}}^D \mu(k) - \sum_{\begin{smallmatrix} k=1 \\ k|D\end{smallmatrix}}^D \mu(k) \right) =-1+1=0.\]
Then, for any $n>p_\jmath$, \[F'_m(n) = -\sum_{\begin{smallmatrix}k=p_\jmath+1\\ (k,m)=1\end{smallmatrix}}^n 1,\] which is the claimed value.
\end{proof}

The game changes then, from approximating the function $\1$ to approximating $-F'_m$ by means of linear combinations of $r_k$. The following seems like a starting point on how to approximate $-F'_m$. We use the standard notation $\varphi$ for the Euler totient function. Define $\nu_\jmath$ as:
\begin{equation}\label{eqn2400}
\nu_\jmath=\frac{\varphi(m)-1}{m-p_\jmath}(r_{p_\jmath}-r_m).
\end{equation}
Notice $\nu_\jmath$ is a reasonable guess firstly in the sense that it is still $m$-periodic and its first $p_\jmath$ values are equal to 0; moreover, $\nu_\jmath(n)$ grows with $n$ at the adequate speed, meaning that $r_{p_\jmath}(n)-r_m(n)$ jumps by $p_\jmath$ at the multiples of $p_\jmath$ and remains constant elsewhere, implying an approximately linear pace of increase for $\nu_\jmath$ up to \[\nu_\jmath (m-1) \approx \varphi(m)-1=-F'_m(m-1).\]
So taking this approximation is a bet on the distribution of numbers relatively prime with $m$, between $p_\jmath+1$ and $m-1$, being equidistributed to some degree.

We will make use of the following standard estimate:

\begin{lem}\label{lemT}
Let $1 \leq n<m = \prod_{i=1}^{\jmath} p_i$. Then, as $\jmath \rightarrow \infty$, we have \[\sum_{\begin{smallmatrix}1\leq k\leq n \\ (k,m)=1\end{smallmatrix}} 1 = n \frac{\varphi(m)}{m} +\mathcal{O}(2^{\jmath}).\]
\end{lem}

\begin{proof} By Lemma \ref{lemZ}, we can see that
\[\sum_{\begin{smallmatrix}1\leq k\leq n \\ (k,m)=1\end{smallmatrix}} 1 = \sum_{k=1}^n \sum_{d | (k,m)} \mu(d) = \sum_{d | m} \mu(d) \sum_{\begin{smallmatrix}1\leq k\leq n \\ d | k \end{smallmatrix}} 1 = \sum_{d | m} \mu(d) \left\lfloor \frac{n}{d} \right\rfloor =: G'(n,m) .\]
Now, in the expression for $G'$ we use the decomposition $\left\lfloor \frac{n}{d} \right\rfloor = \frac{n}{d} - \left\{\frac{n}{d}\right\} = \frac{n}{d}  + O(1)$ to obtain
\[G'(n,m)= n \sum_{d|m} \frac{\mu(d)}{d} - H'(n,m) = \frac{n \varphi(m)}{m} + O(2^\jmath),\]
since $2^\jmath$ is the number of divisors of $m$.
\end{proof}
Notice the relations we used on $G'$ and $H'$, that \[G'(n,m)+H'(n,m) = \frac{n \varphi(m)}{m}\] is valid for all $n,m$.
This Lemma is relevant since the left-hand side in its statement is $1-F'_m(n)$. We would like to be able to apply Lemma \ref{lemT} together with the following:

\begin{prop}\label{lemS} For $n <m$,
\[\nu_{\jmath}(n) + n \frac{\varphi(m)}{m} = O\left(\frac{p_{\jmath}}{\log \log m}\right). \]
\end{prop}

\begin{proof}
Since $n <m$ we know $r_m(n)=n$. Thus, \[(r_{p_{\jmath}}-r_m)(n)= p_{\jmath} \left\{ \frac{n}{p_{\jmath}}\right\} - n = -p_{\jmath} \left\lfloor \frac{n}{p_\jmath}\right\rfloor.\]
Substituting this value in $\nu_{\jmath}(n)$ yields
\[\nu_{\jmath}(n)=  - \left\lfloor \frac{n}{p_{\jmath}}\right\rfloor p_{\jmath} \frac{\varphi(m)-1}{m - p_\jmath}.\]
But we can decompose
\[\nu_{\jmath}(n) + n \frac{\varphi(m)}{m} = p_{\jmath}\left\{ \frac{n}{p_{\jmath}}\right\} \frac{\varphi(m)-1}{m-p_{\jmath}} + \frac{n}{m-p_{\jmath}} - \frac{n\varphi(m)p_{\jmath}}{m(m-p_{\jmath})}.\]
Since $m >> p_{\jmath}  \gtrsim \log m$ and $n<m$, the second term is a number between $0$ and $2$ and thus, clearly contributes at most $O\left( \frac{p_{\jmath}}{\log \log m}\right)$. From \cite[Lemma 4.13 and Theorem 4.15]{BateDia},  we see that  $\varphi(m)/m \approx \frac{1}{\log \log m}$. As the fractional part of a number is bounded by 1, we get the desired estimate for the first term as well. A mixture of all the same arguments gives the same bound for the last term.
\end{proof}

As a closing remark of this section, note that Lemma  \ref{lemS} along with Proposition \ref{lemT} yield that the contributions of some ranges of values of $n$ ($2^{2\jmath}\leq n <m$) to the norm of $\nu_{\jmath}+F'_m$ are unimportant in order to show uniform boundedness. Since $m$ has $\jmath$ different prime factors, when $\jmath$ is large it is to be expected that this unimportant range becomes most of the range of possible values, but good estimates at the first values of $n$ are needed.

\smallskip

Indeed, pointwise convergence is easy to see since $\|F'_m\|_{\alpha} \rightarrow 0$ as $t \rightarrow \infty$ whenever $\alpha <-3$: just estimating $|F'_m(n)|<n-p_{\jmath}$ when $n> p_{\jmath}$ (and 0 otherwise) shows that $\|F'_m(n)\|^2_\alpha \lesssim (p_{\jmath})^{3+\alpha}$, which vanishes as $\jmath \rightarrow \infty$. For large values of $\jmath$, adding $\nu_\jmath$ to the mix does not affect this convergence (its norm is less than that of $r_{p_{\jmath}}-r_m$, which for $\alpha < -3$ goes to 0 as $t \rightarrow \infty$).

\smallskip

The final observation is that one doesn't need to take $r_{p_{\jmath}}$ and in fact it makes sense to make use of all the divisors of $m$ to approximate the function $\sum_{\begin{smallmatrix}1\leq j\leq n \\ (j,m)=1\end{smallmatrix}} 1$.

%%%%%%%%%%%%%%%%%%%%%

\section{An extension of the work by Waleed Noor} \label{section 5}

The main aim of this final section is pushing further the recent approach of Waleed Noor \cite{WaleedNoor} to the context of weighted Dirichlet spaces in order to provide zero-free regions for the Riemann $\zeta$ function.

\smallskip

Let us denote by $\Psi$ the correspondence between $D_\alpha$ and $\ell^{2}_\alpha$ introduced in Subsection \ref{preliminares}, namely:

$$\begin{array}{lccc}
\Psi:\, & \ell^{2}_\alpha  & \longrightarrow & D_\alpha\\
\noalign{\medskip}
& \{ a_{k} \}_{k\geq 1}  & \longrightarrow &  \displaystyle \Psi\Big ( \{ a_{k} \}_{k\geq 1} \Big ) =\sum_{k=1}^{\infty} a_k z^{k-1}
\end{array}$$

\medskip

We have shown in Theorem \ref{thmA} that for $\alpha \in (-3,-2]$, the statement that the Riemann zeta function has no zeros on the band $\{z\in \mathbb{C}:\; \Re (z) > -\frac{\alpha+1}{2}\}$ can be derived provided we establish the existence of some linear combinations of $\{\Psi(r_k): k\geq 2\}$ approximating
$$
f_0(z):=\frac{1}{1-z}, \quad (z\in \D)
$$
in the $D_\alpha$ norm.

\smallskip

Waleed Noor's work \cite{WaleedNoor} addresses an analysis of the linear operator $T: D_{-2} \rightarrow D_0$ defined by
$$
Tf(z) = \frac{((1-z)f(z))'}{1-z}, \quad (z\in \D).
$$
In particular, he proves that such operator is bijective and bounded and identifies $T^{-1}(f_0)$ as well as $\{T^{-1}(\Psi(\{r_k(n)\}_{n=1}^{\infty})): k \geq 2\}$.

\smallskip

In order to extend previous results to $D_{\alpha}$,  we introduce the operators $T_a$ and $T_{a,h}$ where $a >0 $ and $h$ is a function as follows:
\[T_af(z) = \frac{((1-z)^{a} f(z))'}{(1-z)^{a}}, \]\[ T_{a,h}f(z) = \frac{((1-z)^{a} f(z))'}{(1-z)^{a}} \cdot h(z).\]

In particular, $T=T_1$; if $f_1\equiv 1$, $T_{a}=T_{a,f_1}$. We also have that $T_{a,h} = M_h \cdot T_a$ whenever such composition is defined (where $M_g$ denotes the operator of multiplication by the function $g$).

\begin{prop}\label{noor}
Let $\alpha \in [-3,-2)$ and $a >0$. The operator $T_{a}: D_{\alpha+2} \rightarrow D_\alpha$ is bijective and bounded if and only if $a \geq -\frac{1+\alpha}{2}$. Moreover, if $h$ is a bounded function in $D_\alpha$ with $|h(z)|>C>0$ for all $z \in \D$ and some constant $C$, the same holds for
$T_{a,h}$.
\end{prop}

%\begin{rem}
%In the particular case of $T_a$ the requirement on $h$ is automatically true.
%\end{rem}

\begin{proof}
We start with the \emph{boundedness}: we know from \cite{WaleedNoor} that $T$ is bounded from $D_{0}$ to $D_{-2}$, but the same proof yields boundedness $D_{\alpha+2} \rightarrow D_{\alpha}$. Now, $T_a = D - a M_{\frac{1}{1-z}}$, where $D$ is the derivative. From the definition of the spaces $D_\alpha$, given in terms of coefficients, it is easy to see that the derivative is a bounded operator from $D_{\alpha+2}$ to $D_{\alpha}$. Thus, $M_{\frac{1}{1-z}}$ must be bounded as the difference between two bounded operators $D-T$. Hence, $T_a$ is bounded. Now, $T_{a,h} = M_h T_a$ and the operator $M_h$ is bounded in  $D_{\alpha}\rightarrow D_{\alpha}$ for any bounded function: $H^{\infty}$ coincides with the set of multipliers in all Dirichlet spaces larger than $H^2$ \cite{Taylor}. That also gives that $M_{1/h}$ is a bounded operator $D_{\alpha}\rightarrow D_{\alpha}$.

\smallskip

Let us now see that $T_a D_{\alpha+2}=D_{\alpha}$. Suppose $\alpha \in[-3,-2)$, $a>0$ and notice as before that $T_a = D-aM{\frac{1}{1-z}} : D_{\alpha+2} \rightarrow D_\alpha$. Denote, for $k \geq -1$, \[t_k= (k+1)(k+2)^{\frac{\alpha}{2}} \mbox{ and }  d_k = \frac{(k+1)^{a-\alpha+\frac{1}{2}} }{a-\alpha+\frac{1}{2}}.\]
Then as $k$ becomes large we have \begin{equation}\label{eqn10001} d_k - d_{k-1}= (k+1)^{a-\alpha-\frac{1}{2}} ( 1 +o(1)).\end{equation}

We define functions $\{g_k\}_{k \in \N} \in D_{\alpha+2}$ (in fact, polynomials) by
\[g_k(z) = \frac{1}{t_k} \left( z^{k+1} - \frac{\sum_{n=0}^k (d_n-d_{n-1}) z^n}{d_k} \right)\]
for any $z \in \D$. Notice that $\{g_k\}_{k \in \N}$ is a basis of $(1-z) \mathcal{P}$  ($\mathcal{P}$ denotes the class of all polynomials). Thus it spans a dense subspace of $D_{\alpha+2}$. We are going to check the following claims:
\begin{enumerate}
\item[(A)] $\|g_k\|^2_{\alpha+2} \geq 1$ and $\|g_k\|^2_{\alpha+2} = 1 + o(1)$.
\item[(B)] $\{g_k\}_{k\in \N}$ is a Riesz basis for $D_{\alpha+2}$.
\item[(C)] $\|T_a g_k\|^2_{\alpha} \geq 1$ and $\|T_a g_k\|^2_{\alpha} = 1 + o(1)$.
\item[(D)] $\{T_a g_k\}_{k\in \N}$ is a Riesz basis for $D_{\alpha}$.
\end{enumerate}

Hence $T_a$ is a surjective operator (and in fact, any two different elements of $(1-z)D_{\alpha+2}$ will have different images).

Let us now prove the claims. We can compute directly \[\|g_k\|^2_{\alpha+2} = \left(\frac{k+2}{k+1}\right)^2 + \frac{1}{d_k^2} \sum_{n=0}^k (d_n-d_{n-1})^2 (n+1)^{\alpha+2},\] already showing that $\|g_k\|^2_{\alpha+2} >1$ indeed. Substituting \eqref{eqn10001} yields \[\|g_k\|^2_{\alpha+2} = 1 + O(\frac{1}{k+1}).\] This shows that (A) holds.
To check (B) we need to start by evaluating $\left\langle g_i , g_j \right\rangle_{\alpha+2}$ for $i>j$: From the definition of $g_i$ and of the inner product, we have
\[\left\langle g_i, g_j\right\rangle_{\alpha+2} = \frac{1}{t_i t_j d_i} \left( - (d_{j+1}- d_j)(j+2)^{\alpha+2} + \frac{1}{d_j} \sum_{n=0}^j (d_n-d_{n-1})^2 (n+1)^{\alpha+2}\right),\]
which, as $j$ grows has the asymptotic expression \[(1+o(1)) \cdot (i+2)^{\alpha+2} (a-\alpha+\frac{1}{2})(j+2)^{a+\frac{1-\alpha}{2}} \frac{-(a+\frac{3}{2})}{2a-\alpha+2},\] where we used that $2a  -\alpha +2$ is positive.
Next, we can take a linear combination $\sum_{\ell=0}^k a_\ell g_\ell$ and make sure that the contribution of the crossed terms in the expression of its norm is smaller than that of the diagonal terms. We start by using the expression we obtained for the inner products:
\[\left| 2 \Re{\sum_{\ell=0}^k \sum_{j=0}^{\ell-1} a_\ell\bar{a_j}\left\langle g_\ell, g_j\right\rangle_{\alpha+2}}\right| \leq C_0 \sum_{\ell=0}^k \sum_{j=0}^{\ell-1} \left(|a_\ell|^2+ |a_j|^2\right) (1+o(1)) (j+2)^{a+\frac{1-\alpha}{2}} (\ell+2)^{\frac{\alpha-3}{2}-a},\] where \[C_0= \frac{(a-\alpha + \frac{1}{2})(a+\frac{3}{2})}{2a-\alpha+2}.\]
Now, since $a +\frac{3-\alpha}{2} >0 > -a+\frac{\alpha-1}{2}$ this is bounded by
\[(1+o(1)) \frac{(a-\alpha+\frac{1}{2})(a+\frac{3}{2})}{(a+\frac{1-\alpha}{2})(a+ \frac{3-\alpha}{2})}\sum_{\ell=0}^k |a_\ell|^2 .\]
Since we know that $\|g_k\|^2_{\alpha+2} \geq 1$, we can then use that $\sum_{\ell=0}^k |a_\ell|^2  \leq \sum_{\ell=0}^k |a_\ell|^2\|g_\ell\|^2_{\alpha+2}$. To complete the proof of (B) we then need to check that \[(a-\alpha+\frac{1}{2})(a+\frac{3}{2}) < (a +\frac{1-\alpha}{2})(a+\frac{3-\alpha}{2})\] but this inequality simplifies to $\alpha^2+2\alpha> 0$ which is true for all $\alpha<-2$. Thus $\{g_k\}_{k \in \N}$ is a Riesz basis for $D_{\alpha+2}$.

The remainder of the claims (C) and (D) will follow a similar logic, where we start by evaluating $T_ag_k$: Since $T_a=D-aM_{\frac{1}{1-z}}$, for $z \in \D$ we have
 \[T_ag_k(z) = \frac{1}{t_k} \left( (k+1)z^k + \frac{a \sum_{n=0}^k z^n d_n  - \sum_{n=1}^k (d_n-d_{n-1})nz^{n-1}}{d_k}\right).\]
Grouping terms of the same order yields
 \[T_ag_k(z) = \frac{1}{t_k} \left( (k+1+a)z^k + \frac{ \sum_{n=0}^{k-1} z^n (a d_n  - (n+1)(d_n-d_{n-1}))}{d_k}\right).\]
Notice that this is a basis of $\mathcal{P}$, which is dense in $D_\alpha$.
The coefficient of $z^n$ in the above expression can be replaced by using 
\begin{equation}\label{eqn10002} 
a d_n  - (n+1)(d_n-d_{n-1})=(1+o(1))\frac{\alpha-\frac{1}{2}}{a-\alpha+\frac{1}{2}}(n+1)^{a - \alpha + \frac{1}{2}}.
\end{equation} 
Replace \eqref{eqn10002} to obtain the estimate on $\|T_ag_k\|^2_{\alpha}$ we need:

\[\|T_ag_k\|^2_\alpha = \left(\frac{k+1+a}{k+1}\right)^2 \left( \frac{k+2}{k+1} \right)^{-\alpha} + d_k^{-2} \cdot (1+o(1)) \frac{(\frac{1}{2} - \alpha)^{2}}{(a-\alpha+\frac{1}{2})^2} \sum_{n=0}^k (n+1)^{2a-\alpha+1}.\] The term arising from $z^k$ already contributes some number greater than 1 to the norm, and the total result differs from $1$ by some terms comparable to $\frac{1}{k+1}$, showing that (C) holds.

Finally, to check the Riesz condition, we assume again that $i>j$ and evaluate the crossed-terms of the inner products by first substituting the value of $d_k$:
\[\left\langle T_a g_i,  T_a g_j\right\rangle_{\alpha}= \frac{(1+o(1)) (\alpha-\frac{1}{2})}{t_i t_j d_i (a- \alpha +\frac{1}{2})} (j+1)^{a+\frac{3}{2}}\left(1+\frac{\alpha-\frac{1}{2}}{a-\alpha+\frac{3}{2}}\right).\]
Then making use of the values of $t_i$, $t_j$, and $d_i$ provides
\begin{equation}\label{eqn10003}\left\langle T_ag_i ,  T_ag_j\right\rangle_{\alpha} = \frac{(1+o(1)) (a+1)(\alpha-\frac{1}{2})}{(a- \alpha +\frac{3}{2})} (j+1)^{a+\frac{1-\alpha}{2}}(i+1)^{-a+\frac{\alpha-3}{2}}.\end{equation}
At this point, we can argue as in (B): notice $\alpha-\frac{1}{2} <0 <a-\alpha+\frac{3}{2}$ and $a+1>0$, and therefore
\[\left|2\Re{\sum_{i=0}^k \sum_{j=0}^{i-1} a_i \bar{a_j} \left\langle T_a g_i , T_a g_j \right\rangle_\alpha }\right|\leq \frac{(1+o(1))(\frac{1}{2}-\alpha)(a+1)}{a-\alpha+\frac{3}{2}} \cdot \left( \sum_{i=0}^k |a_i|^2 \left(\frac{1}{a+\frac{3-\alpha}{2}} + \frac{1}{a+\frac{1-\alpha}{2}}\right)\right).\]

Since we already know that $\|T_ag_i\|_\alpha^2 \geq 1$, we can write this as
\[\left|2\Re{\sum_{i=0}^k \sum_{j=0}^{i-1} a_i\bar{a_j} \left\langle T_a g_i , T_a g_j \right\rangle_\alpha }\right|\leq C_1 (1+o(1))\sum_{i=0}^k |a_i|^2\|T_a g_i\|^2_\alpha,\]
where \[C_1= \frac{(\frac{1}{2}-\alpha)(a+1)(2a+2-\alpha)}{(a-\alpha+\frac{3}{2})(a+\frac{3-\alpha}{2})(a+\frac{1-\alpha}{2})}.\]

We will be finished with showing (D) once we check that $C_1 <1$. Indeed, this is equivalent with
\[(\frac{1}{2}-\alpha)(a+1)(2a+2-\alpha)<(a-\alpha+\frac{3}{2})(a+\frac{3-\alpha}{2})(a+\frac{1-\alpha}{2}),\]
which simplifies to
\[a^3+\frac{5}{2}a^2+ \frac{\alpha^2+7}{4} a + \frac{-2 \alpha^3 + 3 \alpha^2 + 2 \alpha +1}{8}>0.\]
Since the coefficients on $a^3, a^2$ and $a$ are positive numbers and $a$ is itself positive, we just need to guarantee that $H(\alpha):=-2\alpha^3 + 3 \alpha^2 + 2 \alpha +1$ is also positive. Notice that $H'(\alpha)=0$ has no solutions for $\alpha \in[-3,-2]$ and thus the infimum of $H$ over that interval is at one of its extremes, but $H(-3)=76$ and $H(-2)=23$. This completes the proof of the surjectivity of $T_a$. Note that the surjectivity of $T_{a,h}$ follows having in mind that $M_h$ is invertible along with the claims [A-D].

\smallskip

Finally, it remains to see that \emph{$T_a$ is one-to-one}: if it is, then $T_{a,h}=M_h T_a$ is too, since $M_h$ is invertible.
Now, consider two functions $g_1$ and $g_2$ in $D_{\alpha+2}$ such that \[T_a g_1 = T_a g_2.\]
This means that $(1-z)^a (g_1-g_2)$ must be a constant function $c \in \C$ and if $c \neq 0$ we must have $(1-z)^{-a} \in D_{\alpha+2}$, which happens exactly when $a < -\frac{1+\alpha}{2}$. This concludes the proof.
\end{proof}

\begin{rem} Note that when $0 < a < -\frac{1+\alpha}{2}$, only the injectivity fails in the theorem. Our proof in that case implies that $(1-z)D_{\alpha+2}$ is mapped injectively to its image.
\end{rem}

\smallskip

For $k\geq 2$ let $h_k$ be the function introduced in \eqref{eqhk}, that is,
\begin{equation*}
h_k(z)= \frac{1}{1-z} \log \left(\frac{1+z+...+z^{k-1}}{k}\right), \quad (z \in \D).
\end{equation*}
Waleed Noor showed that
\begin{equation}\label{eqnA07}
T h_k = \Psi(r_k).
\end{equation}

In our setting, next results identify the inverse images of the special functions $f_0$ and $\Psi(r_k)$ under some of the introduced operators.

\begin{lem}
Let $g=(1-z)^c$, with $c \neq a$. Then \begin{equation}\label{eqnA08}
T_{a,g} \frac{(1-z)^{-c}}{c-a} = f_0.\end{equation}
In the particular case of $c=a-1 \geq 0$, for $k \geq 2,$ we also have
\begin{equation}\label{eqnA06}
T_{a,g} \left( \frac{h_k}{(1-z)^{c}}\right) =   \Psi(r_k).
\end{equation}
\end{lem}

\begin{rem}
Before proving the result, notice that $(1-z)^{-c} \in D_\beta$ if and only if $c < -\frac{\beta-1}{2}$. Thus for the functions we found $(1-z)^{-c}$ to be in $D_{\alpha+2}$, there is an additional requirement that $c<- \frac{\alpha+1}{2}$, in order to guarantee that the application of the operator is correct. With regards to $h_k (1-z)^{1-a}$ being elements of $D_{\alpha+2}$, notice $h_k \in D_{1-\varepsilon}$ for all $\varepsilon >0$: Waleed Noor \cite[Last line of proof of Lemma 7]{WaleedNoor} showed that these functions have coefficients bounded by a multiple of $k/n$, and thus $h_k(1-z)^{1-a}$ is in $D_\alpha+2$ precisely when $a > \frac{\alpha+3}{2}$. Since we are going to take $a \geq 1$ to apply the Lemma, $a$ is in the correct range of values.
\end{rem}

\begin{proof}
To see \eqref{eqnA08}, call $f(z)= \frac{(1-z)^{-c}}{c-a}$. By direct computation we can see that
\[T_{a,g} f = \frac{\frac{c}{c-a}(1-z)^{-c-1} \cdot (1-z) + \frac{a}{a-c} (1-z)^{-c}}{(1-z)^{1-c}}.\]
The right hand side simplifies to $f_0$.

In order to check \eqref{eqnA06}, fix $c=a-1 \geq 0$, and then we can use \eqref{eqnA07} to see that \[T_{a,g} \left( \frac{h_k}{(1-z)^{a-1}}\right) = \frac{\left( (1-z) \cdot h_k \right)'}{1-z} = Th_h = \Psi(r_k).\]
\end{proof}

Gathering all the previous results, the following extension of Waleed Noor's criterion holds, the previous result being the case $\alpha=-2$, $a=1$:

\begin{thm}
Let $\alpha \in [-3,-2]$, $a \in [1,\frac{1-\alpha}{2})$ and $\Omega_\alpha= \{ z \in \C: \Re(z) > \frac{-1-\alpha}{2}\}$.
The function $(1-z)^{1-a}$ is in the closure in $D_{\alpha+2}$ of $\Span\{(1-z)^{1-a}h_k:  k \geq 2\}$ if and only if $\Omega_\alpha$ is a zero-free region  of the Riemann zeta function.
\end{thm}

\medskip

\section*{Acknowledgements}

The authors would like to thank Professor K. Seip for helpful discussions.

%%%%%%%%%%%%%%%%%%%%%

\end{document}